\documentclass[11pt]{amsart}

\usepackage{amsmath,amssymb,times}
\usepackage{url} 
 \usepackage{hyperref}

\input amssym.def
\input amssym.tex

\newtheorem{theorem}{{Theorem}}[section]
\newtheorem{proposition}[theorem]{{Proposition}}
\newtheorem{isom.ext}[theorem]{{Trivial isometric extension}}
\newtheorem{corollary}[theorem]{{Corollary}}
\newtheorem{fact}[theorem]{{Fact}}
\newtheorem{remark}[theorem]{{Remark}}
\newtheorem{exo}[theorem]{{Exercise}} 

\def\div{\mathrm{div}}

\begin{document}

\title[Homogeneous spaces,   dynamics,  cosmology...]{Homogeneous spaces,   dynamics,  cosmology: 
Geometric flows and rational dynamics}
\author{Abdelghani Zeghib}
\address{CNRS, UMPA, \'Ecole Normale Sup\'erieure de Lyon, 
46, all\'ee d'Italie,
69364 Lyon cedex 07,  France}
\email{abdelghani.zeghib@ens-lyon.fr 
\hfill\break\indent
\url{http://www.umpa.ens-lyon.fr/~zeghib/}}

\date{}
\maketitle
\begin{center}
\end{center}

\begin{abstract}  The   Ricci flow is a parabolic   evolution equation in the space of 
Riemannian metrics
of a smooth manifold.  To  some extent,   Einstein equations  give rise to a similar hyperbolic evolution.  
The present text is an introductory   exposition to   Bianchi-Ricci and Bianchi-Einstein flows, that is, the  restricted finitely dimensional
 dynamical systems, obtained by considering    homogeneous metrics.

\end{abstract}


\section{Introduction}

These notes are  variations around the  homogeneous   space $$X_n  = Sym_n^+=  GL(n, {\Bbb R})/O(n)$$
that is, the space of $n\times n$ positive definite symmetric matrices,  
 and sometimes, its subspace of those matrices with  determinant 1,  $$Y_n =  SSym_n^+= SL(n, {\Bbb R})/SO(n)$$

Besides their striking beauty,  these spaces modelize many structures and support  fascinating geometry 
and dynamics. It is surely very interesting to investigate 
interplays between these aspects. We will not do it systematically here, but rather briefly note some of them\footnote{We would like to emphasis on that this is a preliminary short non-finished work. This explains in particular  why many  proofs are left as exercises. Also, this text may be considered as expository, although the method here does not follow any existing approach.}.
Our modest remark here is that ``Bianchi''  Ricci flows and Bianchi cosmologies are better seen as natural dynamical systems (i.e.  differential equations) on $X_n$, respectively of first and second orders. To be more precise, Bianchi spaces are special homogeneous Riemannian spaces, those given by left invariant metrics on Lie groups.    For a given Lie 
group $G$, the space of such metrics is identified with $X_n$, $n  = \dim G$. The Ricci flow acting of the space of Riemannian metrics on a manifold, becomes here 
a (gradient-like) flow on $X_n$.  Similarly,  the Cauchy problem for the (vaccum) Einstein equations  becomes here a dynamical system on $TX_n$. The group $G$ acts 
preserving all these dynamical systems. 

\subsubsection*{References} There are now abundant references on the Ricci flow techniques, since their  use 
by Perelman in his program on the geometrization conjecture on 3-manifolds. We may quote as an example \cite{Cho} as an interesting recent reference. The ``toy'' homogeneous case (with which we are dealing here) was in particular  investigated  in \cite{Ise, Ise2}. 

The history of  cosmology from a relativistic point of view, i.e. applying Einstein equations, is very old, and was in fact usually considered within a  homogeneous framework, even an isotropic one, as in the standard big-bang model \cite{Haw, Wal, Wei}.  Even, concrete interplay between cosmology and the mathematical theory of dynamical systems, were 
involved in the literature, but this seems to be not well known  by ``mathematicians'' 
(which gives motivation for our text here). As recent references, we may quote \cite{And, Rin, Wai}.

For the general  material on homogeneous spaces,  we quote
\cite{Bes, Che, Mil, One}. 

Finally, in these proceedings, recommended  references would include
\cite{Klo, Sar}.

\section{A multifaceted  space} \label{definitions}
 The general linear group
$GL(n, {\Bbb R})$ acts transitively on the space of positive scalar products on 
${\Bbb R}^n$,  the stabilizer of  the canonical scalar product being the orthogonal 
group
$O(n)$. This space is therefore identified to  
the homogeneous space
$GL(n, {\Bbb R})/O(n)$. It will be sometimes  more convenient to deal with a reduced variant:
$$Y_n = SL(n, {\Bbb R})/SO(n)$$
which then represents conformal scalar products, or equivalently, scalar products with unit volume (i.e. their unit ball has volume 1 with respect to the canonical volume).

For a more intrinsic treatment,  we start with a real  vector space $E$, and consider 
$Sym(E)$ the space of its quadratic forms. Inside it, we have the open subspace 
of all non-degenerate ones $Sym^*(E)$, and $Sym^+(E)$ (or $X(E)$)  the open cone  
of positive definite ones. We also get spaces of conformal structures by taking 
quotient by ${\Bbb R}$ acting by homothety; in particular  $SSym^+(E)$ (or $Y(E)$)
will denote the space of conformal positive definite structures. In the case of 
$E = {\Bbb R}^n$, we use the notations: 
$$Sym_n, \;  Sym_n^*, \;  Sym_n^+ (= X_n) ,   SSym_n^+ (= \; Y_n) $$

They are identified to subspaces of symmetric matrices $\{ A  = A^* \in M_{nn}\}$. The last 
three spaces correspond respectively to:  $\det A \neq 0$, $A$ has positive
eigenvalues, and positive eigenvalues with $\det A = 1$.

\begin{exo} Show that the $GL(n, {\Bbb R})$-action on $Sym_n$ is given
by $g.A =g^*Ag$ (where $g^*$ is the transpose of $g$).

\end{exo}



\subsubsection{A metric on the space of metrics}  \label{metric.metrics}

A Euclidean structure $q$
on a vector space $E$ induces similar ones on associated spaces, in particular on the 
dual $E^*$ and on $E^* \otimes E^*$. 
If $(e_i)$ is a $q$-orthonormal basis, then its dual basis $(e_i^*)$ is also orthonormal, 
and also is the basis $(e_i^* \otimes e_j^*)$.

Now, since  $Sym_n^+$ is open in 
$Sym_n$, its tangent space   at any point $q$ is  
naturally identified to $Sym_n$, which is thus endowed with the scalar product 
$\langle, \rangle_q$. 
Therefore, $Sym_n^+$ becomes (tautologically) 
a Riemannian space.

\begin{exo} Show that

\begin{eqnarray} \label{product}
\langle p, p \rangle_q =  tr(q^{-1}pq^{-1}p)\;  (=  tr (pq^{-1}pq^{-1}))
\end{eqnarray}
 (where 
$p \in T_q (Sym_n^+)$is  identified with a matrix  $\in Sym_n$).

- Show that $Sym_n^+$ is isometric to the product ${\Bbb R}_+^* \times Y_n$, more 
precisely, one has an isometry:  $A \in Sym_n^+ \mapsto (\log \det A,  \frac{A}{\det A}) \in ({\Bbb R}, n  \mbox{can}) \times Y_n$ (where $\mbox{can}$ stands for  the canonical metric
of ${\Bbb R}$).

-- Show that $q \mapsto q^{-1}$ is an isometry, that is,  $Y_n$ as well as $X_n$ are Riemannian  symmetric spaces.

\end{exo}

\begin{exo} For $n = 1$, $Y_n$ is ${\Bbb R}_+^*$ endowed with 
$\frac{dx^2}{x^2}$.

- $Y_2$ is ``a hyperbolic plane", i.e. a homogeneous simply connected surface
with (constant) negative curvature. Compute this constant. Does this correspond
to a classical model of the hyperbolic plane?

- Show that the $SL(n, {\Bbb R})$-action on $Y_n$ factors through a faithful action 
of $PSL(n, {\Bbb R})$. 

\end{exo}

\subsubsection{Symmetric matrices vs quadratic forms} 
\label{matrix.form}

We guess it is worthwhile to 
seize 
the opportunity and  clarify the relationship between quadratic forms and their representations as 
matrices.  Let $E$ be a vector space, and as above $Sym(E)$ and $Sym^+(E)$ its   spaces of 
quadratic forms, and those which are positive definite, respectively.  A basis  $(e_i)$ 
of $E$ yields a matricial representation  isomorphism 
$$P \in Sym(E) \mapsto p = (P(e_i), P(e_j))_{ij} \in Sym_n
( = Sym({\Bbb R}^n) )$$
 Justified by a latter use (see \S \ref{geodesic.flow}), the restriction to $Sym^+(E)$, will be denoted by:  
$Q \in Sym^+(E) \mapsto q \in  Sym_n^+$.

In fact, these representations depend only on the scalar product on $E$ for which the given basis is
orthonormal. 

Now, given $(Q, P) \in Sym^+(E) \times Sym (E)$, that is a pair  of  a scalar product together with a quadratic 
form on $E$,  one associates a $Q$-autoadjoint  endomorphism $f: E \to E$, representing $P$ by means 
of $Q$, that is $P(x, y) = Q(x, f(y)) = Q(f(x), y)$, where   $P$ and $Q$ are understood here as
symmetric  bilinear forms.

\begin{fact} Given $(Q, P)$ and their associates $(q, p)$, the endormorphism 
$f$ 
has  a matrix representation   $A = q^{-1}p$.

Conversely, given $Q$ and a  $Q$-autoadjoint endomorphism $f$, its  corresponding 
quadratic form $P$ has a matrix representation $p = qA$.

\end{fact}

 
 \subsection{Flats} \label{definition.flat}
 
 Let ${\mathcal B} = (e_i)$ be a basis of $E$. 
 The {\bf flat}  $F_{\mathcal B} \subset  Sym^+(E)$  is  the space of scalar products on $E$ for
which ${\mathcal B}$ is orthogonal. It is parametrized by n positive  reals $x_i$, its 
elements have the form:  $\Sigma x_i e_i^*\otimes e_i^*$.

\begin{exo} Show that  the metric induced on $ F_{\mathcal B}$ is given by 
$\Sigma_i \frac{dx_i^2}{x_i^2}$, and that 
$$(t_1, \ldots, t_n)\in ({\Bbb R}^n, \mbox{can}) \mapsto \Sigma \exp(t_i) e_i^*\otimes e_i^*
\in F_{\mathcal B}$$
 is an isometric immersion.

Prove that:

- $F_{\mathcal B}$ is totally geodesic in $Sym^+(E)$.(Hint: make use of the isometries of 
$Sym_n^+$, $p \mapsto \sigma_i^*p\sigma_i $, where $\sigma_i$ is the reflection fixing all 
the $e_j$, $j \neq i$, and $\sigma_i(e_i) = -e_i$).

-  $GL(E)$ acts transitively on the set of flats of the form $F_{\mathcal B}$ (Hint: relate 
this to the simultaneous diagonalization of quadratic forms).

- Any geodesic of $Sym^+(E)$ is contained in some $F_{\mathcal B}$.

\end{exo}


\section{An individual left invariant metric}
\label{individual}

Standard references for this section and the following one are \cite{Bes, Che, Mil, One}.

\subsubsection{Equation of Killing fileds}  Let $(M, \langle, \rangle )$ be a pseudo-Riemannian manifold. A Killing field $X$ is a vector field that generates a (local) flow of isometries. 
Any vector field $X$ has a covariant derivative which is an endomorphism of $TM$ defined by:  
$D_xX: u \in T_xM \mapsto  \nabla_u X (x) \in T_xM$, where $\nabla$ is the Levi-Civita connection 
of the metric.

\begin{fact} $X$ is a Killing field, iff, $D_xX$ is skew-symmetric with respect 
to $\langle, \rangle_x$,  for any $   x \in M$.

\end{fact}

\begin{proof} Let us first recall that this is the case for the  Euclidean space:  in fact, 
this is
  equivalent to  that the Lie algebra of the orthogonal group is the space of skew-symmetric 
matrices. 

In the general case, assume $x$ generic, that is,  $X(x) \neq 0$, and consider 
$N $ a small transverse submanifold (to $X$ at $x$). Let $Y$ and $Z$ be two vector fields
defined on $N$, and extend them on open neighborhood of $x$, by applying 
the flow of $X$, that is by definition: $[X, Y] = [X, Z] = 0$. 

If $X$ is Killing, then 
$\langle Y, Z \rangle$ is constant along $X$: $X. \langle Y, Z \rangle = 0$. Thus, by definition of the Levi-Civita connection, 
$\langle \nabla_X Y, Z \rangle + \langle Y,  \nabla_X Z \rangle = 0$. Apply 
the commutations  $\nabla_XY = \nabla_YX$, $\nabla_XZ =\nabla_ZX$,  to get: 
$0 =   \langle \nabla_YX, Z \rangle + \langle Y,  \nabla_Z X \rangle $, that is 
$D_xX$ is skew-symmetric. 

It is also easy to use those arguments backwards, 
  that is, if  $D_xX$ skew-symmetric for any $x$, then $X$ is a Killing field.
\end{proof}

\subsubsection{Three  Killing fields}
\label{three.killing}

Let now $X, Y$ and $Z$ be  {\bf three Killing fields}. Apply skew-symmetry for all their covariant derivatives,  and  get (at the end of substitutions) the following  formula:

 \begin{equation} \label{Three.Killing}
 2  \langle \nabla_XY, Z \rangle    =   \langle [X, Y],  Z  \rangle  + \langle [Y,  Z], X \rangle  
- \langle [Z, X], Y \rangle 
\end{equation}

\begin{remark} Observe the beauty (= symmetry) and easiness of the formula!

\end{remark}

\begin{remark} {\em
 Recall Koszul's formula for the Levi-Civita
connection:

\begin{eqnarray} \label{Koszul}
2 \langle \nabla_YZ, X \rangle= & &Y \langle Z, X \rangle + Z \langle X, Y \rangle 
- X \langle Y, Z \rangle - \cr
&&\langle Y, [Z, X] \rangle + \langle Z, [X, Y] \rangle + \langle X, [Y, Z] \rangle
\end{eqnarray}

It  implies the previous formula of three Killing fields. 
Conversely, this last formula yields  Koszul's one for any combination of Killing fields with  coefficients
(non-necessarily constant)  functions on $M$. 
 In particular (\ref{Three.Killing}) yields (\ref{Koszul}) 
 if $M$ is homogeneous. In fact, as (\ref{Three.Killing}) holds as $X, Y$ and $Z$
 are pointwise Killing at order 1, (\ref{Three.Killing}) yields (\ref{Koszul}) also 
 in the general non-homogeneous case.

}

 \end{remark}

\subsection{ Left invariant metrics} We are interested now on Riemannian metrics on a Lie group  $G$  which are left  invariant, i.e any {\bf left}  translation 
$x \mapsto gx$ is isometric. This is in particular the case of any flow $\phi^t(x) = g^tx$, where 
$\{ g^t \}$ is a one-parameter subgroup of $G$. Its infinitesimal  generator 
$X(x) = \frac{ \partial \phi^t}{\partial t}(x)\vert_{t=0}$ is a Killing field. 
This is  a {\bf right} invariant vector field:  $X(g) = X(1)g$. 

Therefore, {\it a left invariant metric is
 exactly a metric  admitting the  right 
invariant fields as Killing fields.}

Another characterization is that {\it a left invariant metric is one for which  left  invariant fields have a constant length.} (but they are not necessarily Killing).

Such a metric is equivalent to giving a scalar product on the tangent space of  one point in  $G$, say $T_1G$, i.e. the Lie algebra of $G$. We keep the same notation
$\langle, \rangle$ for both the metric on $G$ and the scalar product on its Lie algebra ${\mathcal G}$.

Here, to be precise, we define the Lie algebra ${\mathcal G}$ {\it as the space of right}
invariant vector field on $G$. (The other choice, i.e. that of left invariant vector fields, would induce modification of signs in some formulae). 

From the formula of three Killing fields, one sees that the connection on $G$ 
is expressed by means of the scalar product on ${\mathcal G}$ and the Lie 
bracket. In other words, one can forget the group $G$ and see all things on 
the Lie algebra. For instance the Riemann  curvature tensor is a 4-tensor 
on ${\mathcal G}$. The Ricci curvature is just a symmetric endomorphism 
of ${\mathcal G}$.

\subsection{The connection}  Any Lie group has a 
canonical torsion free  connection defined for right invariant vector fields by:
 $$\nabla_X Y = \frac{1}{2} [X, Y].$$

 Any other left invariant connection can be written
$\nabla_X Y = \frac{1}{2} [X, Y]Ê+ C(X, Y)$, where $C: {\mathcal G} \times 
{\mathcal G} \to {\mathcal G}$ is a bilinear map, and the connection is torsion free, iff, 
$C(X, Y) =C(Y, X)$.  In the case of the Levi-Civita connection of a metric, one deduces from the  formula of three Killing fields

\begin{eqnarray} \label{connection}
2  \langle C(X, Y), Z \rangle   & =  & 
 \langle [X, Z], Y \rangle + \langle X,  [Y,  Z] \rangle \cr
 & = &  \langle ad_X^* Y, Z \rangle +  \langle ad_Y^* X, Z \rangle 
 \end{eqnarray}
 
 (Here as usually  $ad_uv = [u, v]$, and $ad^*$ is its  adjoint with respect 
 to $\langle, \rangle$).
 
 It then follows:
 
 \begin{equation} \label{connection.adjoint}
  \nabla_X Y = \frac{1}{2} ([X, Y] + ad_X^* Y + ad_Y^* X
  \end{equation}

 \subsubsection{Sectional curvature}
 The following formulae  for various curvatures follow from 
  Eq. (\ref{connection}),  see for instance \cite{Bes} for detailed proofs.

 \begin{eqnarray*}
\langle R(X, Y)Y, X \rangle =& -& \frac{3}{4} \langle [X, Y], [X, Y] \rangle
 - \frac{1}{2} \langle [X, [X, Y],  Y\rangle  \cr
 & -& \frac{1}{2} \langle  [Y, [Y, X],  X] \rangle + \langle C(X, Y), C(X, Y) \rangle \cr
&-&  \langle C(X, X), C(Y, Y)\rangle + \langle Y, [[X, Y], X] \rangle
\end{eqnarray*}


\subsubsection{Ricci curvature} \label{ric.formule1}

The tensor $C$ disappears in the expression of the Ricci and scalar curvatures!

\begin{eqnarray}\label{Ricci1}
Ric(X, X) =& -& \frac{1}{2} B(X, X) -\langle [Z, X],  X\rangle \cr
 & -& \frac{1}{2}\Sigma_i \vert [X, e_i] \vert^2
+\frac{1}{4} \Sigma_{ij} \langle [e_i, e_j], X \rangle^2
\end{eqnarray}
where, 

$\bullet$ $B$ is the Killing form of ${\mathcal G}$, that is the bilinear form   $(X, Y) \mapsto 
tr(ad_Xad_Y)$, 

$\bullet$ 
$(e_i)$ is any  orthonormal basis of $({\mathcal G}, \langle, \rangle)$,

$\bullet$ and finally,    $Z$ is the vector of  ${\mathcal G}$  defined by 
$\langle Z, Y \rangle = trace( ad_Y)$,  that is, it  measures the unimodularity
defect of ${\mathcal G}$ by means of $\langle, \rangle$.

\subsubsection{Scalar curvature} \label{scalar.formule1}

\begin{eqnarray}\label{Scalar1}
 r=  -\frac{1}{4} \Sigma_{ij} \vert [e_i, e_j] \vert^2
-\frac{1}{2} \Sigma_i B(e_i, e_i) - \vert Z \vert^2
\end{eqnarray}

\subsection{Warning: left vs right} \label{Warning} A metric is bi-invariant i.e.  invariant under  both 
left and right translations of  $G$ if and only if  its associated scalar product is 
 invariant under the $Ad$ representation. In this case, the connection is  the canonical one given by 
$\nabla_X Y = \frac{1}{2} [X, Y]$, for $X $ and $Y$ right-invariant vector fields.  As said above, 
  this  torsion free  connection exists on any Lie group, but does not in general  derive from 
a Riemannian or a pseudo-Riemannian metric (i.e. it is not a Levi Civita connection). Its geodesics trough the neutral element are one-parameter groups, and its curvature is given by $R(X, Y)Z =\frac{1}{4} [ [X, Y], Z]$, and has a Ricci curvature
$Ric(X, Y) = \frac{1}{4}B(X, Y)$, where $B$ is the Killing form (recall that a connection, not necessarily pseudo-Riemannian,  has a Ricci curvature, $Ric(X, Y) = tr ( Z \mapsto R(X, Z)Y$).

 \subsubsection{Other quantities} The following fact, left as an exercise, gives characterization of bi-invariant metrics:

\begin{fact}  A left invariant  metric is right invariant (and hence bi-invariant) iff it 
satisfies one of the following conditions:

\begin{enumerate}

\item For any right invariant (Killing) field $X$, $\langle X, X \rangle$ is constant on $G$.

\item The orbits of any such $X$ are geodesic

\item The orbit of $1$ ($\in G$) under any such  $X$ is a    one parameter group.

\end{enumerate} 

\end{fact}

All this suggests the possibility to define other quantities essentially equivalent to the  Ricci curvature. Say, in the 
bi-invariant case, the Ricci  curvature is essentially the Killing form, and hence, in the general  left invariant case,   the remaining
part  $Re$ (of the Ricci curvature)
is an  obstruction to the  constancy of $\langle X, X \rangle$, or (equivalently) an obstruction 
 for   one parameter groups to be geodesic.  A naive construction goes as follows.
For $X $ a right invariant vector field, let $l^X: x \in G \mapsto \langle X(x), X(x) \rangle$, its
length function, and $dl^X_1$ its differential at $1$. Define $Re(X, X) = 
tr ( dl^X_1 \otimes dl^X_1)$...


\vspace{1,5cm}

\section{Curvature mappings on $X_n$ }

\subsection{All scalar products together: The space $Sym^+({\mathcal G})$ et al}  We are now  considering  all left invariant Riemannian metrics on $G$. As said above the space of such metrics  can be identified with
$Sym^+({\mathcal G})$, the space of positive definite scalar products on ${\mathcal G}$. Let 
$Sym({\mathcal G})$ be the space of all quadratic forms on ${\mathcal G}$.  Then, 
the above formula for the Ricci curvature determines a map:
$$Ric: Sym^+({\mathcal G}) \to  Sym({\mathcal G}) $$

\subsection{$Aut({\mathcal G})$-action on $Sym^+({\mathcal G})$} Not only interior 
automorphisms of ${\mathcal G}$, but also ``exterior'' ones, i.e. general automorphisms 
act on $Sym({\mathcal G})$. Their group $Aut({\mathcal G})$ is sometimes  
identified with $Aut(G)$, assuming implicitly that $G$ is simply connected. Of course, this action is compatible with all mappings that  will be discussed below. 


\subsection{All Lie algebras together} 
 We can now  go a step further and deal  with all Lie algebras of a given  dimension $n$. As vector spaces 
 they are identified with ${\Bbb R}^n$. The space of quadratic forms and the  positive definite
 one are $Sym_n$ and $Sym_n^+$. In short, for any Lie algebra ${\mathcal G}$
  (endowed with a basis allowing one to identify it with ${\Bbb  R}^n$), we have a Ricci and a scalar 
  curvature mapping, involving the bracket structure of ${\mathcal G}$:
   $$Ric_{{\mathcal G}}:  Sym_n^+ \to Sym_n$$
 $$r_{{\mathcal G}}:  Sym_n^+ \to {\Bbb R}$$

\subsection{Formulae} 

\label{ric.formule2}
The Lie algebra ${\mathcal G}$ has a basis $(e_i)$.
An element of $Sym_n$ is denoted by  $p =  (x_{ij})_{1 \leq, i, j, \leq n}$. The structure constants  $C_{ij}^k$,  are defined by $[e_i, e_j] = C_{ij}^ke_k$.

 {\it(Here and 
everywhere in this paper, if a letter  is repeated as a lower and an upper index, 
like $k$ is the last equation, we use the Einstein summation convention, for example, 
$a_i^jb_j $ stands for $\Sigma_j a_i^jb_j$, etc...)}.

From  Formula (\ref{Ricci1}),  we have: 
$$Ric_{{\mathcal G}}:  q = (x_{ab}) \in  Sym_n^+ \mapsto p = (X_{ab}) \in Sym_n$$
Where:

\begin{eqnarray} \label{Ricci2}
X_{ab} = -\frac{1}{2} B_{ab} - \frac{1}{2} C_ {ai}^k C_{b j}^l x_{kl}x^{ij}
+\frac{1}{4} C_{ik}^p C_{jl}^q x_{pa}x_{qb}x^{ij}x^{kl} + \frac{1}{2} E_{ab}
\end{eqnarray}

and,  
$B_{ab} = C_{ai}^j C_{bj}^i $
 is the matrix representing the Killing form,  $(x^{ij})$ is the inverse 
matrix of $(x_{ij})$, and 
$$E_{ab}= C_{ij}^jx^{is} (C_{sa}^lx_{lb} + C_{sb}^lx_{la}) $$
This last term  depends on $q$, but vanishes identically  if ${\mathcal G}$ is unimodular.
So, assuming ${\mathcal G}$ unimodular  will simplify  and shorten the formula.

 \subsection{Parameter} To simplify, we will restrict ourselves to unimodular algebras, and so $(E_{ab})$ disappears. 
 Any Lie algebra is defined by a system $(C_{ij}^k)$ which furthermore satisfies
the  Jacobi identity. We can then consider a mapping,
 
 \begin{eqnarray}\label{Ricci3}
 Ric: (\overrightarrow{C}, q) =  ( (C_{ij}^k), (x_{ab}))  \in {\Bbb R}^{n^3} \times Sym_n^+ \mapsto  p = (X_{ab}) \in Sym_n \cr
 X_{ab}
 =   -\frac{1}{2} C_{ai}^j C_{bj}^i  - \frac{1}{2} C_{ai}^k C_{b j}^l x_{kl}x^{ij}
+\frac{1}{4} C_{ik}^p C_{jl}^q x_{pa}x_{qb}x^{ij}x^{kl}
\end{eqnarray}

 This map is equivariant with respect to the $GL(n, {\Bbb R}) \times GL(n, {\Bbb R})$-action on the  source and the $GL(n, {\Bbb R})$-action on the target.

 





\section{Three-dimensional case: Bianchi geometries (of class A)}

A {\bf Bianchi space} (or  geometry) is a homogeneous  Riemannian 3-manifold together with 
a transitive free action of a Lie group. This is therefore equivalent to giving 
a left invariant  Riemannian metric on a 3-dimensional Lie group. These groups have been classified 
by Bianchi (see for instance \cite{Che, Wai}). They split into classes A and B, according to they are
unimodular or not. To simplify we will consider here only class A. 

\subsubsection{Milnor (or Bracket cyclic)  bases of  a 3-dimensional Lie algebra}
\label{definition.milnor}
 Let ${\mathcal G}_{(a, b, c)}$ be the Lie algebra  generated 
by $\{u, v, w\}$ with relations:
 $$[u, v] = a w, [v, w] = bu, [w, u] = cv$$
It is easy to show that this is actually a Lie algebra, i.e. that the Jacobi identity is satisfied.

Conversely, a    basis
${\mathcal B} = \{u, v, w \}$ of  a Lie algebra ${\mathcal G}$ is called  a {\bf Milnor basis}   (\cite{Che}) if
it satisfies the previous relations: $[u, v] = a w, [v, w] = bu, [w, u] = cv$. In particular, ${\mathcal G}$
is then isomorphic to 
   ${\mathcal G}_{(a, b, c)}$.

\subsection{Invariance of  Milnor flats under $Ric$}

\label{ric.flat} Let us recall here that in dimension 3, giving the Ricci curvature is equivalent to giving the full Riemann curvature tensor (in higher dimension, Ricci is too weaker than Riemann).

Recall that the  {\bf flat}  $F_{\mathcal B}$ determined by ${\mathcal B}$ is  the space of scalar products on ${\mathcal G}$ for
which ${\mathcal B}$ is orthogonal. 

We will say that $F_{\mathcal B}$ is a {\bf Milnor flat} if ${\mathcal B}$
is a Milnor basis.

The flat $F_{\mathcal B}$   is parametrized by 3 positive  reals $x, y$
and $z$, where:
$$x = \langle u, u \rangle, y = \langle v, v \rangle, \; 
\mbox{and} \;  z = \langle w, w \rangle.$$

The corresponding metric will be denoted by $(x, y, z) \in ({\Bbb R}^+)^3$

Define the {\bf extended flat} $\overline{F_{\mathcal B}}$ to be the set of all quadratic forms which are diagonal in the basis ${\mathcal B}$, i.e. they have the same form as
elements of $F_{{\mathcal B}}$, but $x, y$ and $z$ are allowed to be any real numbers.

One uses the orthonormal basis $\{ \frac{u}{\sqrt{x}},  \frac{v}{\sqrt{y}}, \frac{w}{\sqrt{z}}  \}$ for the metric $(x, y, z)$, and computes  from Formula (\ref{Ricci1}) (or (\ref{Ricci3})),

\begin{eqnarray}  \label{Ricci.System1}
 \left \{
\begin{array}{c}
Ric(u, u) =  \\ \\
Ric(v, v)= \\  \\
Ric(w, w) = \\ \\
\end{array}
\begin{array}{c}
 \frac{1}{2} ( b^2\frac{x^2}{yz} - a^2\frac{z}{y} - c^2\frac{y}{z}) + ac  \\  \\
  \frac{1}{2} ( c^2\frac{y^2}{xz} - a^2\frac{z}{x} - b^2\frac{x}{z}) + ab  \\ \\
  \frac{1}{2} ( a^2\frac{z^2}{xy} - b^2\frac{x}{y} - c^2\frac{y}{x})  + bc \\ \\
  \end{array}
  \right.
\end{eqnarray}

In fact, $Ric$ is diagonal in the basis ${\mathcal B}$, i.e. $Ric(u, v) = \ldots = 0$.
We can on the other hand perform  some  simplification, for instance, for $Ric(u, u)$ we have: 
\begin{eqnarray*}
\frac{1}{2} ( b^2\frac{x^2}{yz} - a^2\frac{z}{y} - c^2\frac{y}{z}) + ac 
&=&  \frac{1}{2yz} (b^2x^2  - c^2y^2 -a^2z^2 + 2 ac yz) \cr
& =& \frac{1}{2yz} (b^2 x^2 - (cy-az)^2 )
\end{eqnarray*}
\begin{proposition}

Consider a 3-Lie algebra   ${\mathcal G}$, and 
a Milnor flat  $F_{\mathcal B}  \subset Sym_3^+$. 
 Then, the Ricci map  sends:  
$$(x, y, z) \in F_{\mathcal B}  \mapsto (X, Y, Z) \in \overline{F_{\mathcal B}}$$

and is  given by:

\begin{eqnarray} \label{Ricci.System2}
 \left \{
\begin{array}{c}
X =  \\ \\
Y = \\  \\
Z = \\ \\
\end{array}
\begin{array}{c}
\frac{1}{2yz} (b^2x^2  - (cy -az)^2)  \\  \\
\frac{1}{2xz} (c^2 y^2 - (az - bx)^2)  \\ \\
 \frac{1}{2xy} (a^2 z^2 -( bx -cy)^2) \\ \\
  \end{array}
  \right.
  \end{eqnarray}

\end{proposition}

\begin{remark} {\em
Observe the complete symmetry  of these equations: there is for instance 
``a duality'' $x \mapsto b$: everywhere the coefficient of 
$x$ (resp. $x^2$) is $b$ (resp. $b^2$).  Recall that the coordinate $x$ correponds to the vector $u$, and they are both related to the coefficient $b$, by the fact 
that the unique bracket proportional to $u$ is $[v, w] = bu$. The same observation applies to the other variables, following a same correspondence:
$(x, y, z) \mapsto  (b, c, a)$.
}

\end{remark}




\subsection{Scalar curvature} Similarly, from  Formula (\ref{Scalar1}), we infer:

\begin{eqnarray}  \label{scalar.flat}
  r=   \frac{1}{2xyz} (- b^2x^2  - c^2y^2 -a^2z^2 + 2acyz + 2 abxz + 2 bc xy).
  \end{eqnarray}

\section{Structure of unimodular Lie algebras in dimension 3}
\label{unimodular.classification}

\begin{proposition}  Any unimodular 3-Lie algebra has a Milnor  basis. More precisely
$(a, b, c) \in {\Bbb R}^3 \mapsto {\mathcal G}_{(a, b, c)} \in {\mathcal L} $ gives a parametrization of 
the space of unimodular 3-Lie algebras ${\mathcal L}$.  The diagonal action 
of ${\Bbb R}^*{}^3$ on ${\Bbb R}^3$, and that of the permutation group $S_3$
by permutation of coordinates, preserve the  isomorphism classes of algebras. 

\end{proposition}



In fact, this is equivalent to the more precise: 
\begin{corollary}

There is six  isomorphism  classes of unimodular 3-algebras,  represented as 
follows, where below $+$  (resp. $-$) means any positive (resp. negative) number, e.g. $+1$ (resp. $-1$).
\begin{enumerate}

\item $(0,  0, 0)$: the abelian algebra ${\Bbb R}^3$.

\item  $(0, 0, +)$: the Lie algebra of $G = Heis$, the Heisenberg group.

\item $(0, -, +)$:  the Lie algebra of  $G= Euc$,  the group of rigid motions 
of the Euclidean plane (i.e.  the isometry group of the affine
Euclidean plane). It  is  a  semi-direct product ${\Bbb R} \ltimes {\Bbb R}^2$, where ${\Bbb R}$ acts on ${\Bbb R}^2$, by 
$$  \left(
\begin{array}{cc}
  \cos t &  -\sin t \\ 
\sin t &  \cos t 
 \end{array}
\right). $$

\item $(0, +, +)$:  the  algebra of the group $G= SOL$, the group of rigid motions 
of the Minkowski plane (i.e. the Minkowski space of dimension $1+1$), 
$G=   {\Bbb R} \ltimes {\Bbb R}^2$, where ${\Bbb R}$ acts on ${\Bbb R}^2$ by 
$$  \left(
\begin{array}{cc}
  e^t &  0 \\ 
0 & e^{-t} 
 \end{array}
\right).$$

\item $(-, +, +)$: $sl(2, {\Bbb R})$.

\item $(+, +, +)$: $so(3)$.

\end{enumerate}

They are labeled respectively, Bianchi:  $I$, $II$, $VII_0$, $VI_0$, $VIII$ and $IX$. 

 \end{corollary}

\begin{proof} (see \cite{Mil}). 
In dimension 3, there  are  exactly two semi-simple Lie algebras,  $so(3)$ and 
$sl(2, {\Bbb R})$, and there are no semi-simple Lie algebra of dimension 
$\leq 2$. Therefore, if a 3-algebra  ${\mathcal G}$ is not  semisimple, then  it contains no semisimple 
algebra, that is, ${\mathcal G}$ is solvable. 

Assume ${\mathcal G}$ contains an (abelian) ideal isomorphic to ${\Bbb R}^2$. Take 
any supplementary one dimensional subspace (hence a subalgebra)  ${\Bbb R}$. Then, ${\mathcal G}$
writes as a semi-direct product of ${\Bbb R}$ acting on ${\Bbb R}^2$. Since 
${\mathcal G}$ is assumed to be unimodular, this action is via  a one parameter 
group of 
 $SL(2, {\Bbb R})$. These one parameter groups split into: elliptic, parabolic and hyperbolic types. We obtain respectively: $Euc$, $Heis$ and $SOL$.

It suffices therefore to show the  existence of such an ideal  ${\Bbb R}^2$.  For this, let us remark that there exists always an abelian ideal of dimension 1, say  ${\mathcal A} 
\cong {\Bbb R}$. Indeed, if the commutator  ideal has dimension 2, then since it is solvable, its commutator 
has dimension 1 or 0...

Now, since ${\mathcal G}$ acts on ${\mathcal A}$, the Kernel ${\mathcal L}$ has at least dimension 2 and contains ${\mathcal A}$.  Let us  consider here the case $\dim {\mathcal L} = 2$, since
the dimension 3 case is easier. By definition (of the Kernel) this is an ideal, and since 
it has dimension 2 and has a non-trivial center  (it contains ${\mathcal A}$), then  it is abelian.
\end{proof}

 


\section{Summary; further comments}



In this section, we present a general setup where the previous constructions can be defined. In other words, we ``summarize'' how one associates  to a Lie group various rational dynamical systems.

\subsection{A rational map} (\S \ref{ric.formule2}). Let ${\mathcal G}$ be a Lie algebra of dimension $n$, 
with a basis $(e_i)$, such that $[e_i, e_j] = C_{ij}^k e_k$, 
and assume to simplify that it is unimodular. Then,  we have a rational 
map: 

\begin{eqnarray}
   Ric_{\mathcal G} &:&    (x_{ab})  \in  Sym_n^+  \mapsto   (X_{ab}) \in Sym_n \cr
     X_{ab} &=& -\frac{1}{2} B_{ab} - \frac{1}{2} C_{a i}^k C_{b j}^l x_{kl}x^{ij}
+\frac{1}{4} C_{ik}^p C_{jl}^q x_{pa}x_{qb}x^{ij}x^{kl} 
  \end{eqnarray}
as well as a rational  (scalar curvature)  function:
\begin{eqnarray}
r_{\mathcal G}: (x_{ab}) \in Sym_n^+ \mapsto X_{ab} x^{ab}
\end{eqnarray}
(recall that $(x^{ij})$ is the inverse matrix of $(x_{ij})$).
  
  \subsection{$Aut({\mathcal G})$-equivariance} Both $Ric_{\mathcal G}$ and $r_{\mathcal G}$ are respectively equivariant and  invariant 
  under the $Aut({\mathcal G})$-action on $Sym_n$ (identified with $Sym({\mathcal G})$, the space of quadratic forms on ${\mathcal G}$).
  
 
 \subsection{Extensions} Actually, everything extends to left invariant pseudo-Riemannian
 metrics on $G$, or equivalently to $Sym^*({\mathcal G})$ the space of non-degenerate 
 quadratic forms on ${\mathcal G}$. The same formulae allow one to calculate the Ricci and
 scalar curvature of such metrics.  \\
 
 $\bullet$ Now, the formulae may have sense even for some degenerate quadratic forms!  \\
 
 $\bullet$ We can also consider  complex quadratic forms on the complexification 
 of ${\mathcal G}$. They form a space $Sym({\mathcal G}) \otimes {\Bbb C}$
 identified with $Sym_n({\Bbb C})$,  the space of symmetric complex matrices.  In other 
 words, to a Lie group $G$ of dimension $n $ is associated an $Aut(G)$-equivariant  rational transformation 
 $Ric_G$ on $ Sym_n({\Bbb C}) $.   We also have an  $Aut(G)$-invariant  meromorphic function
 $r_G$. \\
 
$\bullet$ $Ric_G$ can be written  as a rational vectorial map $(\frac{P_i}{Q_i})_i:$ 
$ {\Bbb C}^{n(n+1)/2} \to {\Bbb C}^{n(n+1)/2} $ where $P_i$ and $Q_i$ are homogeneous polynomials of 
a same degree $2n$ (on  $n(n+1)/2$ variables). \\

$\bullet$ From this, one gets a rational transformation of the projective space
${\Bbb C}P^{n(n+1)/2 -1} \cong  PSym_n({\Bbb C})$.  We will denote it by ${\bf Ric}_G$, 
to emphasize  that it is the extension to the complex projective space. 
For instance, if $n=3$, then 
we have a rational map on ${\Bbb C}P^5$.  \\

$\bullet$ $Ric_G$ is  equivariant  under scalar multiplication. In fact, as this   follows from 
its defining  formula, $Ric_G$ is polynomial when restricted to 
matrices with  $\det = 1$. Therefore, a representative of  ${\bf Ric}_G$ (on the projective 
space) is the polynomial map $(P_i)_i$  (of degree $2n$).  \\

$\bullet$ ${\bf Ric}_G$ is invariant under the (algebraic) action of (the complexification of) $Aut({\mathcal G})$  on ${\Bbb C}P^{n(n+1)/2 -1} (\cong  PSym({\mathcal G}) \otimes {\Bbb C}$). It then determines a map on the quotient space.  It depends however
 on the meaning to give to such a quotient space (by $Aut ({\mathcal G})$). As an algebraic action, 
it has a poor dynamics, and a nice quotient space can be thus  constructed for it.
 There is in particular 
a notion of ``algebraic quotient''. In dimension $n = 3$, the algebraic quotient 
space has dimension $5-3$,  more exactly, it is  a (singular) compact complex surface $S_G$, say.  We have then associated to a 3-dimensional Lie group
a rational map on a compact complex surface $S_G$.  This    map seems to have a ``poor dynamics'', for instance, it has in general a vanishing entropy. We guess nevertheless that  (other)   ``dynamical  invariants'' of it 
can characterize the group $G$ (i.e. two different  groups have different invariants).  It is also worthwhile  to see what happens in   higher dimension case.

 \subsection{Forget invariance} In the formula  defining ${\bf Ric}_G$, we can 
 consider any system of parameters $(C_{ij}^k)$, not necessarily satisfying the Jacobi identity of Lie algebras. We  obtain  a big family of rational transformations generalizing
 those associated to Lie groups.  In this case, various dynamical types may appear.   
 We think it is worthwhile to investigate the structure of this parameter space, and to understand inside it, the (algebraic) set of Lie algebras, the algebraic actions on it...

  \subsection{Cross sections, Flats}
  \label{definition.cross}
  
   Let us call a cross section ${\mathcal S}$ for $Ric_G$ or 
  ${\bf Ric}_G$ a submanifold in $Sym_n^+ $ (resp.  $P Sym_n({\Bbb C})$) which is 
  invariant under $Ric_G$ (resp. ${\bf Ric}_G$) and such that ${\mathcal S}$ meets any $G$-orbit 
  in a non-empty  discrete set. The last condition implies in particular that 
  ${\mathcal S}$ is transversal (at least in a topological sense)  to the  $G$-orbits. In fact there are weaker variants 
  of this definition which can be useful, in particular in a geometric algebraic context. 
   
     $\bullet$ It is in the case where $G$ is unimodular and has  dimension 3  that cross sections occur easily.  In this case, one can  in fact find them, as affine  (resp. projective) subspaces for $Ric_G$ (resp. ${\bf Ric}_G$) 
     of dimension 3 (resp. 2).  Indeed, as explained in \S \ref{unimodular.classification},  the Lie algebra of  such a group has 
     a Milnor  basis ${\mathcal B} =  \{u, v, w\}$ (\S \ref{definition.milnor}), and the flat  $F_{\mathcal B}$ (\S \ref{definition.flat}), or 
     more formally its ``extension'' $\overline{F_{\mathcal B}}$, i.e. the space of quadratic 
     forms diagonalizable in ${\mathcal B}$, is invariant under $Ric_G$. In order to  see that one 
     gets in this way a cross section,  it remains to show the abundance of Milnor bases as in the following statement,

 \begin{exo} Prove that any quadratic form can be digonalized in some  Milnor basis
 of ${\mathcal G}$. {\em (Hint: this can be done by checking case by case. For instance, for the group $SOL$, its Lie algebra is generated by 
 $X, Y, Z$, with relations $[X, Y] = Y, [X, Z]= -Z$ and $[Y, Z] = 0$. Consider $u = X+ T$, where $T$ belongs the the plane ${\mathcal P}$ generated by $Y$ and $Z$.  Then, the restriction of 
$ ad_u$ on ${\mathcal P}$, satisfies $ad_u^2 = -1$. Choose  $u$ to be   orthogonal to ${\mathcal P}$ (with respect to the given metric). Consider    a  non-vanishing vector $v \in {\mathcal P}$, and let $w = ad_u(v)$. Then $\{u, v, w\}$ is a Milnor basis, because of the fact $ad_u^2 =1$. We claim that $v$ can be chosen such that $w$ is orthogonal to $v$. This is a calculation in the basis $\{Y, Z\}$).
 }

\end{exo}   
  
    $\bullet$ Let us point out the following polynomial presentation of $Ric_G$.  As was said above, $Ric_G$ is invariant under scalar multiplication, that is,  it suffices to consider 
    its restriction on unimodular matrices $SSym_n$, in which case, it becomes polynomial.  
  In particular, from  \S \ref{ric.flat}, ${\bf Ric}_G$ has the following form as
  a cubic  homogeneous  polynomial map: 
\begin{eqnarray}
{\bf Ric}_G: {\Bbb C}^3 \to {\Bbb C}^3
\end{eqnarray}
\begin{eqnarray*}
(x, y, z) \mapsto \frac{1}{2}
(x(b^2x^2  - (cy -az)^2), y(c^2 y^2 - (az - bx)^2),  z(a^2 z^2 -( bx -cy)^2) ).
\end{eqnarray*}

  \subsection{Formula in each case}

    \subsubsection{Case of $SO(3)$:}   $a= b = c =1$. 
    $$Ric(x, y, z)=
\frac{1}{2}(x(x^2  - (y -z)^2), y(y^2 - (z - x)^2),  z(z^2 -( x -y)^2) ).  $$

       \subsubsection{Case of $SL(2, {\Bbb R})$:}  $a= b=1$, and  $ c =-1$. 
     $$Ric(x, y, z)=
\frac{1}{2}(x(x^2  - (y +z)^2), y(y^2 - (z - x)^2),  z(z^2 -( x +y)^2) ).  $$ 

\subsubsection{Case of  the Heisenberg group $Heis$:}  $a=b=0, c= 1$
$$Ric(x, y, z)=
\frac{1}{2} y^2(-x, y,  -z).  $$ 

\subsubsection{Case of $Euc$:} $a=0, b = -1, c=1$

$$ Ric(x, y, z) =   \frac{1}{2}(x(x^2-y^2), y(y^2-x^2), -z(x+y)^2  )$$

   \subsubsection{Case of $SOL$:} $a=0, b = c=1$

$$ Ric(x, y, z) = \frac{1}{2}(x(x^2-y^2), y(y^2-x^2), -z(x-y)^2  ) $$


 \subsection*{}

  \subsection{Bianchi-Ricci flow}  Recall that the Ricci flow  associated to a 
  compact manifold $M$ (of finite volume) is an evolution equation on its space $Met(M)$ of Riemannian metrics:
  $$ \frac{ \partial g_t}{\partial t} = - 2 Ric(g_t) + 2\frac{<r(g_t)>}{n}g_t$$
   where $n = \dim M$ and  $<r(g)> = \int r(g)dv_g / Vol(M, g)$ is the  average scalar curvature of $g$ \cite{Che}.

\subsubsection{The vector field ${\mathcal Ric}_G$} Now, if  $G$ is an $n$-dimensional Lie group, then this gives a classical differential 
 equation on the space of its left invariant Riemannian metrics, where one takes 
 a punctual value of the  scalar curvature instead of its average (since this scalar curvature  is constant).  Equivalently, this is 
 a vector field on $Sym_n^+$. In fact, all this is derived from our previous rational map
  $Ric_G$. Since $Sym_n^+$ is an open set in the vector space $Sym_n$, the vectorial 
  map $Ric_G: Sym_n^+ \to Sym_n$ can be alternatively  seen  as a {\bf vector field} on $Sym_n^+$, say,   ${\mathcal Ric}_G$. 
  
  Observe that ${\mathcal Ric}_G$ is invariant under the $Aut(G)$-action on $Sym_n^+$ (which is equivalent to the fact  that $Ric_G$ is equivariant under the (linear) action of $Aut(G)$).

   Let us denote a generic point of $Sym_n^+$ by  
  $q$, and consider the radial vector field 
  ${\mathcal V} (q) = q$. The previous differential 
  equation, which we will call the Bianchi-Ricci flow associated to $G$ is the vector field 
  $$-2 {\mathcal Ric}_G + 2\frac{r_G}{n} {\mathcal V}$$

   \subsubsection{Commutation} Consider the bracket 
   $[{\mathcal Ric}_G, {\mathcal V}] =  D_ {\mathcal V}Ric_G  - 
   D_{Ric_G} {\mathcal V} $, where $D_u$ denotes the  usual derivation in the $u$-direction. This equals   $0- Ric_G$,  since,  $D_{\mathcal V}Ric_G = 0$, i.e. 
  the  map $Ric_G$ is invariant under multiplication; 
   and $D{\mathcal V} = Identity$, everywhere.
 Therefore,   $ [ {\mathcal Ric}_G, {\mathcal V}] = - {\mathcal Ric}_G$.   Because of this
 commutation rule (that is,  the two vector fields generate a local action of the affine group),  the essential dynamics of   the  Bianchi-Ricci flow comes from the 
 ${\mathcal Ric}_G$-part.   
 
  \subsubsection{Bianchi-Hilbert-Ricci flow}
  \label{definition.hilbert}
  
   The remark   applies to any combination of 
 ${\mathcal Ric}_G$ and ${\mathcal V}$: understanding one combination allows one 
 to understand the others. A  famous  one is ${\mathcal Ein} = {\mathcal Ric}_G - 
 \frac{r}{2} {\mathcal V}$, which can be called in this context the ``Bianchi-Einstein flow'', since 
  the tensor $Ric(g)- \frac{r}{2}g$ of a Riemannian manifold 
$(M, g)$ is   called Einstein tensor (this is,  essentially, the  unique combination of $Ric(g)$ 
and $g$ which is divergence free). However, in order to prevent confusion 
with ``Einstein equations'' and some related flows which will be considered below, ${\mathcal Ein}$ could 
be better called Bianchi-Hibert flow.  Indeed, the function 
$${\mathcal H}:  p \in Sym_n^+ \mapsto r(p) \sqrt{ \det (p)} \in {\Bbb R}$$
is the substitute of the classical Hilbert action in the case of left invariant metrics. Indeed:

\begin{exo} Show that ${\mathcal Ein}$ is a gradient vector field. More exactly, 
${\mathcal Ein} = \nabla {\mathcal H}$, where the gradient $\nabla$ is taken 
with respect to the metric of $Sym_n^+$.

\end{exo}

\begin{remark} {\em 
The computation can be  handled in a more  explicit way on a Milnor  flat $F_{\mathcal B}$
(\S \ref{ric.flat}), where the Hilbert action has the form:
$$ {\mathcal H}(x, y, z) =   \frac{1}{2 \sqrt{xyz}} (- b^2x^2  - c^2y^2 -a^2z^2 + 2acyz + 2 abxz + 2 bc xy) $$
and the metric is $$ \frac{dx^2}{x^2} + \frac{dy^2}{y^2} + \frac{dz^2}{z^2}  $$

}
\end{remark}

 \subsubsection{Restriction on $SSym_n^+$}  The interest of the normalization 
 in the definition of the Ricci flow is to let it preserving 
  the volume of the Riemannian metric, that is,  the total  volume remains constant under evolution.  In the case  of left invariant metrics, this is equivalent to the fact  that the vector field ${\mathcal Ric}_G - \frac{r}{n}{\mathcal V}$ is tangent to 
 $SSym_n$. This in turn is equivalent to the fact, that for any  $q \in SSym_n$, 
 $Ric_G(q) -\frac{r (q)}{n}q$ is trace free, which follows from the very  definition of the scalar curvature $r$. This allows one 
 to justify the following simplification: write equations assuming $q \in SSym_n$, i.e. $\det(q) = 1$,  which gives polynomial equation. However, in order to  keep this polynomial natural, do not 
 take reduction of variables
 from the equation $\det(q) =1$. To be more concrete, consider 
 a flat  $F_{\mathcal B}$, then instead of the rational forms of $Ric_G(x, y, z)$ and 
 $r_G(x, y, z)$, we assume $xyz =1$ which leads to polynomial forms: \S \S \ref{ric.flat} and  Formula \ref{scalar.flat} (but we do not 
 go further and eliminate one variable, say $z = \frac{1}{xy}$). We can then 
 write the Bianchi-Ricci flow as follows 
 
  \begin{equation}
 - 2 {\mathcal Ric}_G + 2\frac{r}{3} {\mathcal V} (x, y, z)=   \left \{
\begin{array}{c}
  x(  \frac{2}{3} bx (-2bx +cy+ az) + \frac{2}{3}c^2y^2 + \frac{2}{3} a^2z^2 -  \frac{4}{3} acyz ) \\  \\
    y(  \frac{2}{3} cy (-2cy +bx + az) + \frac{2}{3}b^2x^2 + \frac{2}{3} a^2z^2 -  \frac{4}{3} bcxz)  \\ \\
    z(  \frac{2}{3} az (-2az+ bx +cy) + \frac{2}{3}b^2x^2 + \frac{2}{3} c^2y^2 -  \frac{4}{3} bcxy)  \\ \\
  \end{array}
  \right.
  \end{equation}

   \subsubsection{Differential equations on a projective space} In the same way, we associate to a Lie group $G$ an $Aut(G)$-invariant one dimensional complex algebraic foliation 
   on the projective space $PSym_n({\Bbb C})$. Here, among combinations of the vector fields  ${\mathcal Ric}_G$ and ${\mathcal V}$, 
 only  ${\mathcal Ric}_G$ is relevant, since the radial vector field ${\mathcal V}$  becomes trivial on the projective space.  In the case of  a unimodular 3-group, 
   we have the following homogeneous cubic differential system on ${\Bbb C}^3$:

\begin{equation}
\left\{
\begin{array}{c}
\frac{dx}{dt} =  \\ \\
\frac{dy}{dt}= \\  \\
\frac{dz}{dt}= \\ \\
\end{array}
\begin{array}{c}
  x(b^2x^2  - (cy -az)^2)  \\  \\
y((c^2 y^2 - (az - bx)^2)  \\ \\
z (a^2 z^2 -( bx -cy)^2) \\ \\
  \end{array}
  \right.
\end{equation}

     \subsubsection{Dynamics, compactifications}  It is the dynamics of the Bianchi-Ricci flow $ - 2 {\mathcal Ric}_G + \frac{2r}{3} {\mathcal V}$ which was investigated in the literature \cite{Che, Ise}. As we argued above, this  is essentially the same as that of the Einstein-Hilbert 
     field $\nabla {\mathcal H}$.    But, as a gradient flow, its dynamics 
     is completely trivial on $Sym_n^+$... The point is to study the behavior of orbits 
     when they 
     go to an infinity boundary  $\partial_\infty Sym_n^+$. There is however several ways to  attach such a boundary to (the non-positively curved Riemannian  symmetric space) $Sym_n^+$.
     One naturally wants to interpret ideal points as collapsed   Riemannian metrics. With  
     respect to this, the Hadamard compactification seems to be the most pertinent (see for instance \cite{Klo2}). On the other hand, the advantage  of algebraic compactifications (e.g. the projective space) 
     is to extend the dynamics...





\vspace{1,5cm}

\section{Hamiltonian dynamics on $Sym_n^+$}
\label{geodesic.flow}

After  consideration of some maps and vector fields, we are now going 
to   study  
second order differential equations on $Sym_n^+$,   the prototype of which is the geodesic flow of $Sym_n^+$, and then 
the ``Einstein flow'' associated to a Lie group.

\subsection{Geodesic flow}
Write the metric on $Sym_n^+$ as:  $L(q, p) = \langle p, p \rangle_q = tr(q^{-1}pq^{-1}p)$. Since $Sym_n^+$ is open in 
$Sym_n$, its tangent bundle trivializes $TSym_n^+ =
Sym_n^+ \times Sym_n$. We will use the usual notations
$\frac{\partial L}{\partial q}$, $\frac{\partial L}{\partial p}$ for the horizontal 
and vertical differentials $d_pL$ and $d_qL$.

We have:
 $$\frac{\partial L}{\partial q} (\delta q) = tr (-q^{-1} (\delta q) q^{-1} p q^{-1} p - q^{-1}pq^{-1} (\delta q) q^{-1} p) = -2 tr( (\delta q) q^{-1}p q^{-1}p q^{-1})$$
where $\delta q$ is a horizontal tangent vector, i.e. an element of $Sym_n$.

$$\frac{\partial L}{\partial p} (\delta p) =
 2 tr( (\delta p) q^{-1}p q^{-1}).$$

Now, write:  $q = q(t), \; p(t)= \dot{q} = \frac{\partial q}{\partial t}$, and
compute
$$  \frac{\partial}{\partial t}  \frac{\partial L}{\partial p} (\delta p)= 
 2 tr(  (\delta p) q^{-1}[  - 2 \dot{q} q^{-1}\dot{q} + \ddot{q}]q^{-1}   ).  $$

The Euler-Lagrange equation  is obtained by taking 
$\delta q = \delta p = A$, and writing  for any $A$, 

$$  \frac{\partial}{\partial t}  \frac{\partial L}{\partial p} A -
   \frac{\partial L}{\partial q} A = 0. $$
   
   This reads: 
    $$ tr(  2A q^{-1}( - \dot{q} q^{-1}\dot{q} + \ddot{q})q^{-1}   ) = 0, \; \forall A \in
     Sym_n. $$
and therefore, 

\begin{fact} The equation of geodesics of $Sym_n^+$
is the second order matricial equation on $Sym_n^+$: 
    $$  \ddot{q}   =  \dot{q} q^{-1}\dot{q} $$
or equivalently (in the  phase space):

\begin{equation}
 \left \{
\begin{array}{c}
 \dot{q} =  \\
   \dot{p} = \\ 
\end{array}
\begin{array}{c}
  p   \\
 pq^{-1}p.  \\
  \end{array}
  \right.
  \end{equation}

\end{fact}


\subsection{Other pseudo-Riemannian and Finsler metrics  on $Sym^+_n$} There is a canonical $GL(n, {\Bbb R})$-invariant form $\omega$ on $Sym^+_n$:
$$\omega_q(p) = tr(q^{-1}p).$$
We can then associate to any reals $\alpha$ and $\beta$ a Lagrangian:
$$L_{\alpha, \beta}(q, p)= \alpha (\omega_q(p))^2 +\beta\langle p, p \rangle_q= 
\alpha (tr(q^{-1}p))^2 + \beta tr (q^{-1}pq^{-1}p). $$
For generic $\alpha$ and $\beta$,  this is a homogeneous  pseudo-Riemannian metric, but it can degenerate for
some values. 

  Similarly, there are homogeneous Finsler metrics:
$$F_{\alpha, \beta}(q, p)= \alpha \omega_q(p) +\beta \sqrt{\langle p, p \rangle_q}= 
\alpha (tr(q^{-1}p)) + \beta \sqrt{tr (q^{-1}pq^{-1}p}. $$

\begin{exo} Write the Euler-Lagrange equation for  $L_{\alpha, \beta}$ and 
$F_{\alpha, \beta}$.

- Solve the geodesic equation for $Sym_2$.

\end{exo}

\vspace{1,5cm}

\section{Einstein Equations in a Gauss gauge}
${}$

\subsection*{Cylinders} Let $M$ be a differentiable $n$-manifold endowed 
with a family of Riemannian metrics $g_t$, $t $ is a "time" parameter lying in  an interval $ I$. 
Consider the Lorentz manifold $\bar{M} = I \times M$ endowed with the metric 
$$\langle, \rangle =  \bar{g}= -dt^2 + g_t, \;  \mbox{i.e.}\;  \; \bar{g}_{(t, x)} = -dt^2 + (g_t)_x$$
Such a structure is sometimes called a cylinder. 
 Our purpose is to relate 
geometric (e.g.  curvature) quantities on $\bar{M}$ and $M$. For a fixed point 
$(t, x)$, $R$, $Ric$, and $r$ will denote the Riemann, Ricci and scalar curvatures of 
$(M, g_t)$ at $x$ and $\nabla$ its Levi-Civita connection. The corresponding quantities for $\bar{M}$ are noted by 
$\bar{R}$, $\bar{Ric}$ and $\bar{r}$ and $\bar{\nabla}$.

\subsection{Second fundamental form} The (scalar) second fundamental form of $\{t\} \times M$
is denoted $k_t$ (or sometimes simply $k$).
Actually, the second fundamental form is defined as a vectorial form:  $II(X,  Y)$ equals
the orthogonal projection of $\bar{\nabla}_XY$ on ${\Bbb R}e_0$, where 
$e_0 = \frac {\partial}{\partial t}$. 

The scalar second fundamental form
is defined by
$$   k(X, Y) = \langle II(X, Y), e_0 \rangle =  \langle \bar{\nabla}_{X} Y, e_0 \rangle.$$

The Weingarten map $a =  a_{e_{0}}$ is defined by:
$$ a(X) = -  \bar{\nabla}_{X} e_0.$$
We have:
$$   k(X, Y) =   \langle \bar{\nabla}_{X} Y, e_0 \rangle =  X \langle Y, e_0 \rangle
 -   \langle \bar{\nabla}_{X} e_0, Y \rangle = 0 + \langle a(X), Y \rangle.$$

In other words, $a$ is the symmetric endomorphism associated to $k$ by means
of the metric $g$ (we will use sometimes the notation $a_t$ as well as $g_t$ and $k_t$, 
in order to emphasize  the dependence on $t$).(The definition of $k$ and $a$  coincides with that in the Riemannian case. The unique difference is that   here, $II = -k e_0$, since $e_0$ is unit timelike,  i.e. 
$\langle e_0, e_0 \rangle = -1$).

\subsection{Geometry of the product} 
Consider $e_1, \ldots, e_n$ a frame of vector fields on $M$, that we also 
consider as horizontal vector fields on $\bar{M}$. By definition, they commute
with $e_0 (= \frac {\partial}{\partial t}$).

\begin{fact} We have:
\begin{eqnarray}
 \bar{\nabla}_{e_0} e_0 &=& 0 \; (\mbox{the trajectories of}\;  e_0 \;  \mbox{are geodesic}), \\
  k_t &= & (-1/2)\frac{\partial}{\partial t}g_t, \\ 
 \langle \bar{R}(e_0, e_i)e_i , e_0 \rangle & = &  \frac{\partial}{\partial t} \langle a_t(e_i), e_i \rangle    + \langle a_t^2(e_i),   e_i \rangle,  \\
 \bar{Ric}(e_0, e_0)& =& \frac{ \partial}{\partial t} tr (a_t) + tr(a_t^2). 
\end{eqnarray} \\

\end{fact}

\begin{proof}

$\bullet$  We have $$0 = \partial / \partial t \langle e_0, e_i \rangle = \langle \bar{\nabla}_{e_0} e_0, e_i \rangle + \langle e_0,  \bar{\nabla}_{e_0} e_i  \rangle.$$

But $$\langle e_0,  \bar{\nabla}_{e_0} e_i  \rangle = (1/2) e_i.\langle e_0 ,  e_0  \rangle=0, $$ since 
$e_0$ and $e_i$ commute. Therefore $ \langle \bar{\nabla}_{e_0} e_0, e_i \rangle = 0$, 
$\forall i$. \\


$\bullet$  We have
 $$\frac{\partial}{\partial t}g_t(e_i, e_j) =  e_0 \langle  e_j, e_j \rangle =
  \langle \bar{\nabla}_{e_0} e_i, e_j \rangle + \langle \bar{\nabla}_{e_0} e_j, e_i \rangle$$
  
Since $e_0$ commutes with $e_i$ and $e_j$, this also equals:
  $$   -  \langle (a(e_i), e_j \rangle - \langle a(e_j), e_i \rangle =  - 2 k_t(e_i, e_j) $$

$\bullet$ {\it Computation of $ \bar{R}(e_0, e_i)e_i, e_0 \rangle $Ê}:  Because of the commutation relations, and because  $e_0$ is geodesic, we have, by definition of the curvature:
$$\bar{R}(e_0, e_i)e_0 = - \bar{\nabla}_{e_0} a(e_i),$$
 and thus: 
$$\langle \bar{R}(e_0, e_i)e_0 , e_i \rangle =  - \langle  \bar{\nabla}_{e_0} a(e_i), e_i \rangle=  - e_0 \langle a(e_i), e_i \rangle    -  \langle a(e_i),   a(e_i) \rangle,$$
and since $a$ is symmetric, this also equals:
$$ - \frac{\partial}{\partial t} \langle a(e_i), e_i \rangle    -  \langle a^2(e_i),   e_i \rangle$$

And hence, 
$$\langle \bar{R}(e_0, e_i)e_i , e_0 \rangle  =   \frac{\partial}{\partial t} \langle a_t(e_i), e_i \rangle    + \langle a_t^2(e_i),   e_i \rangle$$


$\bullet$ We can assume that at a fixed point $(t, x)$, the basis 
$(e_i)_{i \geq 1}$ is orthonormal, and taking the sum (over $i >0$) we get: 
$$\bar{Ric}(e_0, e_0) = \frac{ \partial}{\partial t} tr (a_t) + tr(a_t^2).$$
\end{proof}

\begin{remark} {\em In fact, the meaning of ``Gauss gauge'' is nothing but that 
$e_0$ is unit and has geodesic orbits.

}
\end{remark}

 \subsection{Gauss equation} It describes the  relationship between the sectional 
 curvatures for $R$ and $ \bar{R}$:
  \begin{eqnarray} \label{Gauss}
  \langle \bar{R}(e_i, e_j)e_j, e_i \rangle =  \langle R(e_i, e_j)e_j , e_i \rangle 
+ k(e_i, e_i)k(e_j, e_j) \cr
 - k(e_i, e_j) k(e_i , e_j)
\end{eqnarray}
(observe this difference of sign  of the $k$-term, in comparison with the Riemannian case).

\subsection{Einstein evolution equation for $k_t$}

Again, assume  $(e_i)$ orthonormal,    fix $i$, and take the sum over $j >0$. We first have:
$$ \Sigma_j k(e_i, e_i)k(e_j, e_j)=  k(e_i, e_i)tr(a)=  tr(a)\langle(a(e_i), e_i)\rangle  $$ 
and 
$$\Sigma_j  k(e_i, e_j)k(e_i , e_j) =  \langle a^2(e_i), e_i) \rangle$$
(Indeed in matricial notations, $a_{ij} = a_{ji} =  k(e_i, e_j)$, and thus $(a^2)_{ii} = \Sigma_j a_{ij}a_{ji}$).

Therefore, if we consider the quadratic form $l$,  defined by: 
$$l(e_i, e_i) = \Sigma_j (k(e_i, e_i)k(e_j, e_j)  - k(e_i, e_j) k(e_i , e_j))$$ then its associated endomorphism is:
 $$tr(a)a - a^2$$

$\bullet$  $ \bar{Ric}(e_i, e_i)$  equals the trace of 
$u \to R(u, e_i) e_i$.  Remember, $(e_i)$ is a Lorentz orthonormal 
basis, i.e. $\langle e_i, e_j \rangle = 0$, for $i \neq j$, 
$\langle e_0, e_0 \rangle = -1$ and $\langle e_j, e_j \rangle= +1$, 
for $j > 0$.  It then follows that 
$$   \bar{R}(e_i, e_i) =   
   \sum_{j>0} \langle \bar{R}(e_i, e_j)e_j , e_i \rangle
-  \langle \bar{R}(e_0, e_i)e_i , e_0 \rangle$$

$\bullet$ Returning to the Gauss equation (\ref{Gauss}), and taking the sum over $j>0$, we get: 
$$ \bar{Ric}(e_i, e_i) +  \langle \bar{R}(e_0, e_i)e_i , e_0 \rangle = Ric(e_i, e_i)  +  \langle (tr(a)a - a^2) (e_i), e_i \rangle$$

$\bullet$ Replacing $ \langle \bar{R}(e_0, e_i)e_0 , e_i \rangle$ by its previous value:
$$\bar{Ric}(e_i, e_i) +  \frac{\partial}{\partial t} \langle a(e_i), e_i \rangle    +  \langle a^2(e_i),   e_i \rangle =  Ric(e_i, e_i)  +  \langle (tr(a)a - a^2) (e_i), e_i \rangle $$

Equivalently, for any $X, Y \in TM$:
 $$  \frac{ \partial}{\partial t} k_t (X, Y) = -  \bar{Ric}(X, Y)    + Ric(X, Y)  +  \langle (tr(a_t)a_t  - 2a_t^2) (X),  Y\rangle.
  $$

\begin{fact} Define the square power   ${k_t}_{g_t}k_t $ to be the quadratic form associated by means 
of $g_t$  with the matrix $a_t^2$ (where $a_t$ is the matrix associated to $k_t$ via $g_t$).  Then: 
 $$\frac{ \partial}{\partial t} k_t  = -  \bar{Ric}   + Ric  + tr_{g_t}(k_t)k_t  -   2{k_t}_{g_t}k_t.  $$

\end{fact}

\subsection{Gauss constraints} Consider again the  Gauss equation and take the sum over 
$i, j >0$: 
$$\bar{r} - 2 \bar{Ric}(e_0, e_0) = r + (tr_{g_t}k_t)^2 - tr({k_t}_{g_t}k_t) =  r + (tr_{g_t}k_t)^2 -  \vert k_t\vert^2_{g_t}. $$

\subsection{Matrix equations} We are  now going to  write equations by means of symmetric matrices 
associated to the quadratic forms $g_t$ and $k_t$ (see \S \ref{matrix.form}). For this, we fix $x$ and a time $t_0$
and choose
an orthonormal basis 
  $(e_i(t_0) )$ of $T_xM$. We denote by  $q_t$ (resp. $p_t$) the matrix associated with $g_t$
(resp. $-2k_t)$, and by $\bar{ric_t}$ and $ric_t$ (or simply $\bar{ric}$ and $ric$) those
associated with $\bar{Ric}$ and $Ric$ (recall they are the Ricci  curvatures of respectively $\bar{M}$, at $(t, x)$, 
and $(M, g_t)$,  at $x$). With this, we have: \\


$\bullet$  Evolution equations: 

\begin{equation}
  \left \{
\begin{array}{c}
 \dot{q} =  \\ \\
   \dot{p} = \\ 
\end{array}
\begin{array}{c}
  p   \\ \\
 - \bar{ric} + ric + \frac{1}{4}tr(q^{-1}p)p -  \frac{1}{2}q^{-1}pq^{-1}p  \\
  \end{array}
  \right.
  \end{equation} \\

$\bullet$  Gauss constraints (actually said {\bf Hamiltonian constraints})  

\begin{eqnarray*}
\bar{r} - 2 \bar{ric}(e_0, e_0) &= &r + (tr(q^{-1}p))^2 -  tr(q^{-1}pq^{-1}p) \\
&=& r+ tr(q^{-1}p)^2 - \langle p, p \rangle_q \\
&=& r-  L_{-1, 1}(q, p)
\end{eqnarray*} 
where $L_{-1, 1}$ is the pseudo-Riemannian metric  defined in \S \ref{geodesic.flow}, for 
the value $(\alpha, \beta) = (-1, 1)$

\section{Bianchi cosmology}

We will now restrict ourselves to  the 
vacuum  case,  i.e. $\bar{M}$ is Ricci-flat:  $\bar{Ric} = 0$, and thus also 
$\bar{r} = 0$. 
We will also assume $M$ is a Lie group $G$
and the metrics on it (i.e. $g_t$) are left invariant.  Therefore such 
a metric is identified with an element $q \in Sym_n^+$ 
($n = \dim G$, the identification of $Sym^+({\mathcal G})$ with 
$Sym_n^+$ comes from a choice of a basis). Now, $ric$ becomes 
a map $ric: Sym_n^+ \mapsto Sym_n$, and $r: Sym_n \mapsto {\Bbb R}$. We get the
(beautiful) ODE system  with constraints:  

\begin{equation}
 \left \{
\begin{array}{c}
 \dot{q} =  \\ \\
   \dot{p} = \\ \\
L_{-1, 1}(q, p) =  \\
\end{array}
\begin{array}{c}
  p   \\ \\
   ric(q) + \frac{1}{4}tr(q^{-1}p)p -  \frac{1}{2}q^{-1}pq^{-1}p  \\ \\
r(q)\; \;   \mbox{(Hamiltonian constraint)} \\
  \end{array}
  \right.
  \end{equation}

\begin{remark} Observe that $ric$ and $r$ are basic functions,  they depend only on 
$q$ (and not on $p$).

\end{remark}

\subsection{Isometric $G$-action on $\bar{M}$} Here $\bar{M} = I \times G$, 
with $\bar{g} = -dt^2 +g_t$. A left translation $x \in G \mapsto  hx$ is isometric 
for all the metrics $g_t$, and therefore is isometric for $\bar{g}$ as well.

\subsection{The Bianchi-Einstein flow along and on a flat} 
\label{bianchi.einstein}
Actually, there are other constraints to add to the 
ODE system above, in order to get what we will call the  {\bf Bianchi-Einstein flow}.  These
(momentum) constraints will be considered below. Before, let  us consider a subsystem of it, the restriction (of everything) to a Milnor  flat $F_{\mathcal B}$. The following proposition derives from  Formulae   (\ref{Ricci.System2}) and (\ref{scalar.flat}).

\begin{proposition} Let $F_{\mathcal B}$ be a Milnor flat, and $TF_{\mathcal B}$
its tangent bundle, a point of which is denoted by $(q, p)$, 
$q = (x, y, z)$, $p= (x^\prime, y^\prime, z^\prime)$.
 The {\bf Bianchi-Einstein} flow on  $F_{\mathcal B}$ is the following 
system of ODE on $TF_{\mathcal B}$, together with one  algebraic constraint defined  
by a Lorentz metric on $TF_{\mathcal B}$ and a basic function on $F_{\mathcal B}$.
(The phase space has thus dimension 5, and is a fiber bundle over $F_{\mathcal B}$):

 \begin{equation}
 \left \{
\begin{array}{c}
\dot{x} =  \\ \\
\dot{y} = \\  \\
\dot{z} = \\ \\
\dot{x^\prime} =  \\ \\
\dot{y^\prime}  = \\ \\
\dot{z^\prime} =   \\ \\
\end{array}
\begin{array}{c}
x^\prime \\ \\
y^\prime \\ \\
z^\prime \\ \\
 \frac{1}{2yz} (b^2x^2  - (cy -az)^2)  + \frac{1}{4}( \frac{x^\prime}{x} + \frac{y^\prime}{y} + \frac{z^\prime}{z} ) x^\prime - \frac{1}{2} \frac{{x^\prime}^2} {x^2}\\  \\
 \frac{1}{2xz} (c^2 y^2 - (az - bx)^2) + \frac{1}{4}( \frac{x^\prime}{x} + \frac{y^\prime}{y} + \frac{z^\prime}{z} ) y^\prime - \frac{1}{2} \frac{{y^\prime}^2} {y^2} \\ \\
 \frac{1}{2xy} (a^2 z^2 -( bx -cy)^2)  + \frac{1}{4}( \frac{x^\prime}{x} + \frac{y^\prime}{y} + \frac{z^\prime}{z} ) z^\prime - \frac{1}{2} \frac{{z^\prime}^2} {z^2}\\ \\
  \end{array}
  \right.
  \end{equation}

The phase space is a hypersurface (maybe singular) $N$ (in $TF_{\mathcal B}$) defined by  the        Hamiltonian equation:   \begin{eqnarray}
l_{(x, y, z)} (x^\prime, y^\prime, z^\prime) =  - \frac{r(x, y, z)}{2}
\end{eqnarray}
Where $l$ is the Lorentz metric (on $F_{\mathcal B}$):
 \begin{equation}
  l_{(x, y, z)} (x^\prime, y^\prime, z^\prime) =      \frac{x^\prime y^\prime }{xy} + \frac{x^\prime z^\prime}{xz}
+ \frac{y^\prime z^\prime}{yz}, 
 \end{equation}
and $r$ is given by Formula (\ref{scalar.flat}):
\begin{equation*}
  r(x, y, z)= 
   \frac{1}{2xyz} (- b^2x^2  - c^2y^2 -a^2z^2 + 2acyz + 2 abxz + 2 bc xy) 
 \end{equation*}

 
\end{proposition}

\begin{exo} Show explicitly  that the constraint   is preserved by the dynamics, i.e. the vector field determined by the differential equations is tangent to 
the ``submanifold'' $N \subset T F_{\mathcal B}$ defined by the constraint. 

\end{exo}


\subsection{Codazzi (or Momentum) constraints}

The Codazzi equation establishes a relation between the intrinsic and extrinsic  curvatures of a submanifold $M$ in a Riemanniann manifold $\bar{M}$, and is in fact valid in the general
background of pseudo-Riemannian manifolds provided the induced metric on the submanifold is also pseudo-Riemannian, i.e. it is not degenerate.  More precisely, it states that some ``partial symmetrisation'' of the covariant derivative of the second fundamental form (all this depends only upon  data on $M)$ equals the normal part of the Riemann
curvature tensor (this depends on $\bar{M}$). The equation gives    obstructions 
for a (vectorial)  2-tensor to  be  the second fundamental form of a submanifold. 

In the case where $M$ is a  {\bf CMC}  spacelike hypersurface (i.e. 
with a constant mean curvature) in a  Ricci flat Lorentz manifold
$\bar{M}$,   one can deduce from Codazzi equation, by taking a trace, that 
the second fundamental form  $k$ is a divergence free 2-tensor.  This applies 
in particular to our case: our hypersurfaces are $G$-orbits and thus  are CMC.

Let us recall some definitions. Firstly, if $k$ is a symmetric 2-tensor on $M$, then 
its covariant derivative $\nabla_X k$  with respect to a vector   $X$,  is a 2-tensor:

$$(\nabla_X k) ( Y, Z) =   X k(Y, Z) - k(\nabla_XY, Z) - k(Y, \nabla_XZ) $$

Now $\div k$ is a 1-form, the trace of $\nabla k$ (with respect to the metric of $M$), i.e. if $(e_i)$ is an orthonormal 
basis: $$\div k (X) = \Sigma_i \nabla_{e_i}k(e_i, X)$$

\subsubsection{Divergence of left invariant quadratic forms on Lie groups} At first glance
one can guess that left invariant objects are divergence free (with respect to left invariant Riemannian metrics). This is however false (apart from some trivial cases).

Let $G$ be a 3-dimensional unimodular  Lie group, endowed 
with a left invariant metric $ \langle, \rangle = q \in Sym^+({\mathcal G})$, with a Milnor $q$-orthonormal 
basis $\{u, v, w\}$: $[u, v]= aw, [v, w] = bu$ and $[w, u] = cv$ (see  \S \ref{definition.milnor}). 
The proof of the following facts and corollaries is left as exercise. 
Let $ p \in Sym({\mathcal G})$ represent a left invariant quadratic 
form.

\begin{fact} Let  
$X, Y$ and $Z$ be right invariant vector fields, with $X (1) = e \in {\mathcal G}$. 
Let $g^t = \exp t e$. 
Then the derivative $X. p(Y, Z)$ at $1 \in G$ is given by:
$$X.p(X, Y) = \frac{\partial} {\partial t} p(Ad  g^t(Y), Adg^t(Z)) =
p([X, Y], Z) + p(Y, [X, Z))$$

\end{fact}

\begin{fact} For the basis $\{u, v, w\}$, we have:

$\bullet$ $\nabla_u u = \nabla_v v = \nabla_w w = 0$, 

$\bullet$ $2 \nabla_u w =  (-c + a-b)v$, 
$2\nabla_v w = (b-a+c)u$ ... (Use Formula (\ref{Three.Killing}))


\end{fact}

\begin{corollary} Consider the left invariant quadratic form,  $p_{12} = du \otimes dv + dv \otimes du$. Then:\\
$\bullet$

   \begin{eqnarray*}
  u.p_{12}(u, e) + v.p_{12}(v, e) + w.p_{12}(w, e) & =& -c-a , \mbox{for}\;  e= w \\
& =& 0, \; \mbox{for}\;   e = u, \; \mbox{or}\;   e=
v
\end{eqnarray*}
 
$\bullet$  $p_{12}(u, \nabla_u w) + p_{12}(v, \nabla_v w) + p_{12}(w, \nabla_w w)= 0$

$\bullet$ It then follows  that $\omega = \div p_{12}$ is such that
$\omega (u) = \omega (v) = 0$, and $\omega (w) = -(c+a)$, that is 
$\omega = -(c+a) dw$.

\end{corollary}

\begin{corollary}   Let us say  a Milnor basis is {\bf generic}
 if 
$(a+c)(a+b)(b+c) \neq 0$. Then, for a generic Milnor basis,  any divergence free left invariant quadratic form (with respect to 
the metric for which this basis is orthonormal) is diagonalizable in this basis. 

In other words (keeping the  previous notation), along a Milnor flat $F_{\mathcal B}$, 
an element $p \in T_qSym_3^+$ satisfies the momentum constraints, iff, 
$p \in T_q (F_{\mathcal B})$ (or in more linear words,   $p \in \overline{F_{\mathcal B}}$)

\end{corollary}






\subsection{Cross sections for the Bianchi-Einstein  flow} On $TSym_3^+$, the group $Aut(G)$ acts, 
 preserving  the  Bianchi-Einstein flow (determined by $G$). A  cross
section (\S \ref{definition.cross}) will play the  role of a flow on a  quotient space (for the $Aut(G)$-action).

\begin{proposition} The Bianchi-Einstein flow on a generic Milnor flat is a cross section 
of the full Bianchi-Einstein flow (with constraints) on $Sym_3$ endowed with the $Aut(G)$-action. Generic flats exist  except in the abelian and nilpotent  cases,  i.e. 
when $G$ is ${\Bbb R}^3$ or the Heisenberg group $Heis$.
\end{proposition}

\begin{proof}   Firstly, one easily sees that if there exists a Milnor basis 
for which $a $ and $b$ $\neq 0$, then after re-scaling,  this basis becomes generic. 
This exists exactly when $G $ is different from $ {\Bbb R}^3,$ and  $ Heis$.  If we are not in these 
cases, then we can assume, after re-scaling if necessary, that all Milnor bases are generic.  Let ${\mathcal B}$ be  such a basis, 
and  $(q, p) \in TSym_n^+$, then up to application of an element 
of $Aut({\mathcal G})$, 
$q \in F_{\mathcal B}$. But since ${\mathcal B}$ is generic,   if 
$p$ satisfies the momentum constraints, then $p \in TF_{\mathcal B}$, which means 
that $TF_{\mathcal B}$ is a cross section.   \end{proof}

\subsubsection{Case of $G = {\Bbb R}^3$} In the case of the  Heisenberg group, there is exactly one momentum constraint which gives rise to invariant sets of the system. There is
no such constraint in the case of ${\Bbb R}^3$, where we obtain the following system:

\begin{equation}
 \left \{
\begin{array}{c}
 \dot{q} =  \\ \\
   \dot{p} = \\ \\
L_{-1, 1}(q, p) =  \\
\end{array}
\begin{array}{c}
  p   \\ \\
   \frac{1}{4}tr(q^{-1}p)p -  \frac{1}{2}q^{-1}pq^{-1}p  \\ \\
0\; \;  ( \mbox{the lightlike cone bundle of } L_{-1, 1}) \\
  \end{array}
  \right.
  \end{equation}
  
  The spacetime $\bar{M} $ has a metric 
  $$\bar{g} = 
  -dt^2 + t^{2p_1} du^2 + t^{2p_2} dv^2 + t^{2p_3} dv^2.$$
  
  This  is called a {\bf Kasner} spacetime (observe that in some cases, e.g. $p_1 = p_2 = p_3$, this  is just the Minkowski space) \cite{And, Bog, Wei}.
  
  \begin{exo} Prove the previous form of $\bar{g}$ and solve the same
  problem in the case of the Heisenberg group.
  
  \end{exo}

\subsection{Isometry group of $\bar{M}$} As said previously, the left action 
of $G$ on itself induces, by definition of its metric, an isometric
action on  $\bar{M}$. In fact, if for some level $(M, g_t)$, there are extra-isometries 
(i.e. other than left translations), then they extend to $\bar{M}$.  More precisely, 
if the metric at some  level, say $t= 0$, is identified with $q_0 \in Sym^+({\mathcal G})$, and 
$\bar{M}$ corresponds to a point $(q_0, p_0) \in T Sym_n^+ = Sym_n^+ \times Sym_n$, 
and $K \subset Aut({\mathcal G})$ is the stabilizer of $(q_0, p_0)$, then, on the one hand,  $K$ acts as an isometric isotropy group for $(M, g_0)$ ($g_0$ corresponds to $q_0$). 
On the other hand, $K$ preserves the Bianchi-Einstein trajectory of $(q_0, p_0)$, and 
thus acts isometrically on $\bar{M}$ (as isotropy  for any point identified 
with $1 \in G$).


\subsection{An example: Bianchi  $IX$}   This means
$G = SO(3)$, or more precisely its universal cover the sphere  $S^3$.  In this case, there are 
Milnor bases with
$a= b = c =1$. Any other Milnor basis 
satisfies these equalities, up to re-scaling. Also, 
 all such bases are equivalent up to conjugacy and re-scaling. Yet,  this is 
 the most challenging case of Bianchi cosmologies  (see for instance \cite{Rin}). As an example, 
  {\bf TAUB-NUT} spacetimes are exact solutions  of  the Bianchi-Einstein equations of class $IX$.
 They are characterized among Bianchi $IX$ spacetimes as those having 
  extra-symmetries, i.e. 
  a non-trivial isotropy, which then must be $SO(2)$ (and thus their  isometry group is 
  $S^3 \times SO(2)$, up to a finite index). Nevertheless, their high complexity (at least among exact solutions) led people to describe them as ``counter-examples 
  to everything''!  In a Milnor  flat where $a =b = c = 1$, these spacetimes 
correspond   (up to isometry)    to $x= y$, and $x^\prime = y^\prime$.
The left invariant  metric on $G$ (at any time) corresponds to 
a Berger sphere, i.e. (up to isometry) a metric on the   sphere derived from
the canonical one,  by 
rescaling the length along the fibers of a Hopf fibration. In other words, the set of  solutions of the 
Bianchi-Einstein flow, which are Berger spheres at any time, is  closed  and invariant,  say, the
TAUB-NUT set.


\subsection{Effect of a non-vanishing cosmological constant} Instead of requiring  
$\bar{M}$ to be Ricci-flat,  let us merely assume it to be Einstein, i.e. $\bar{Ric}  =  \Lambda \bar{g}$. 
Its effect is essentially an additive constant (related to $\Lambda$) in all equations
and constraints.  This situation does not seem to be systematically investigated  
in the literature. In particular, one can wonder  whether the introduction of 
$\Lambda$ is ``catastrophic''  or in contrary produces only a  moderate effect. A similar 
situation is that of the paradigmatic example in holomrophic dynamics, of the 
quadratic family  $z \mapsto z^2 + c$.  Here the variation of the parameter $c$ generates 
a chaotic dynamics as well as  a fractal geometry \cite{Bea}.

\subsection{Wick rotation} Here $\hat{M} = I \times M$
is endowed with the Riemannian metric  $\hat{g} = +dt^2 +g_t$. Writing 
$\bar{Ric} = 0$, yields: 

\begin{equation}
\left \{
 \begin{array}{lcl}
  \frac{\partial}{\partial t}g_t &=& -2k_t\\ \\
\frac{ \partial}{\partial t} k_t  &= &  - Ric  + tr_{g_t}(k_t)k_t  -   2{k_t}_{g_t}k_t  \\ \\
0 & = & 
 r +  \vert k_t\vert^2_{g_t} -  (tr_{g_t}k_t)^2\; (\mbox{Constraint})
 \end{array}
 \right.
\end{equation}

For instance,  this allows one to construct examples of Riemannian Ricci flat manifolds, 
of co-homogeneity 1, i.e.   their isometry group has codimension 1 orbits.
 
 Notice  that the (true) Einstein equations (i.e. without symmetries) can not be solved
 in a Riemannian context (they cannot be transformed to a hyperbolic PDE system).  Maybe, this Bianchi situation can give insights on the reasons 
 behind this fact.
 
 Finally, it does not seem there exists a ``true Wick rotation'', i.e. some correspondence 
 between   solutions of Bianchi-Einstein equations in the Lorentzian and Riemannian cases.
 (Compare with \cite{B-B}).

\subsection{Orthonormal frames approach vs Metric approach}
 As a result of a search on fundamental 
references in this area,  ``dynamical systems and cosmology'', one can get at least 
\cite{Bog, Rya} and 
\cite{Wai} which are surely the most known  and recent  synthesis in this ``emerging'' domain. The authors adopted there an ``orthonormal frames approach''
in opposite to our ``metric approach'' here (see explanations therein). They obtained 
the following system of   quadratic  polynomial  differential equations
on ${\Bbb R}^5$. 
 
 \begin{equation} 
 \begin{array}{lcl}
\Sigma_+^\prime&=&-(2-q)\Sigma_+-\mathcal{S}_+\\
\Sigma_-^\prime&=&-(2-q)\Sigma_--\mathcal{S}_-\\
N_1^\prime&=&(q-4\Sigma_+)N_1\\
N_2^\prime&=&(q+2\Sigma_++2\sqrt{3}\Sigma_-)N_2\\
N_3^\prime&=&(q+2\Sigma_+-2\sqrt{3}\Sigma_-)N_3
\end{array}
 \end{equation}
 
where: 
$$\begin{array}{lcl}
\mathcal{S_+} &=&\frac{1}{6}\left[\left(N_2-N_3\right)^2-N_1\left(2N_1-N_2-N_3\right)\right] \\ \\ 
\mathcal{S_-} &= &\frac{1}{2\sqrt{3}}\left(N_3-N_2\right)\left(N_1-N_2-N_3\right) \\ \\
q& = &\frac{1}{2}(3\gamma -2)(1-K)+\frac{3}{2}(2-\gamma)(\Sigma_+^2+\Sigma_-^2) \\ \\
K &= & \frac{1}{12}\left[N_1^2+N_2^2+N_3^2-2\left(N_1N_2+N_2N_3+N_3N_1\right)\right]
\end{array} $$

Here $\gamma$  is a parameter: 
$2/3<\gamma<2$. (See for instance  \cite{Sar}).

\subsubsection{Comparison} This system of differential equations   must  be  ``equivalent'' to 
our equations (\S \ref{bianchi.einstein}) on the tangent bundle of a Milnor flat, which was a rational differential system
on  ${\Bbb R}^6$ with one constraint.  A formal definition of equivalency of approaches
is that the two systems are ``bi-rationally equivalent''. However, the transformation of our system to this polynomial system is by no means obvious.  This last system was not a priori motivated by simplifying our more ``naive'' one, but rather by considering  another point of view in considering Einstein equations. Instead of studying the evolution with time of the metrics
on spacelike slices, one  considers the evolution of brackets of {\it orthonormal} frames 
on these slices. The gauge freedom is more subtle in this case, but still  this 
method is very clever, as shown by the simplified form of the equations here. In our Bianchi case, i.e. where spacelike slices are Lie groups with left invariant metrics, one can very roughly say that the bi-rational equivalence  comes from
the projection map $ Mil({\mathcal G}) \to Sym({\mathcal G})$, where $Mil({\mathcal G})$
is the space of Milnor bases of ${\mathcal G}$.  The next step is to lift the Einstein equation
(including a gauge choice) to $Mil({\mathcal G})$ (more precisely an associated bundle) and to  take  the quotient 
by the $G$-action!

\subsection{Further remarks}  This beauty of $Sym_n$ appeals one to go beyond..., but as
we said, our contribution here is  essentially preliminary and expository. 
Let us mention some 
facts that were  not considered here (with the hope to give details on some of
them in the future). 
\subsubsection{Variants of $Sym_n^+$ } First, one can generalize the discussion from $Sym_n^+$ to 
$Sym_n^*$, the space of all pseudo-Euclidean products, i.e. non-degenerate quadratic forms. Everything extends there,  a pseudo-Riemannian metric (on 
$Sym_n^*$), Ricci maps, Bianchi-Einstein flows... 

The components of $Sym_n^*$ are spaces of quadratic forms of a given signature. As 
for $Sym_n^+$,  each  component is a pseudo-Riemannian symmetric space and plays 
a universal role in its class.  \\

 $\bullet$ {\it Complex case.} The same is true for complex spaces:  $Sym_n^*({\Bbb C})$,  the space of complex non-degenerate quadratic forms on ${\Bbb C}^n$, is  a holomorphic symmetric space... \\
 
$\bullet${\it Projectivization.} Taking the associated projective spaces will send all these spaces  into  compact ones, and hence 
compactify them, by attaching various boundaries, with  more or less nice  interpretations.  A natural requirement is that ideal points correspond 
to collapsing of Riemannian metrics, say in the Gromov sense \cite{Gro} (restricted here to homogeneous spaces).  By algebraicity, all differential equations extend 
to the  projective spaces.  \\

$\bullet${ \it Fiberwise constructions.} If $E \to B$ is a vector fiber bundle, then one can 
associate to it $Sym^+(E)$...  \\

$\bullet$ {\it Configuration spaces.}  Another interesting  aspect  of $Sym_n^+$ is
its configuration space aspect.   We mention here the case of `` hydrodynamics'', where 
a geometric formalism (a Riemannian metric, its geodesic flow...) was 
developed (see for instance \cite{Arn}) following similar ideas as those presented here.  There are also  other
 non-linear and infinite dimensional situations, in particular a space $Sym^*(E)$
 associated to a Hilbert space $E$ could be exciting!

\subsubsection{Geodesic flows of left invariant metrics} For a Lie group $G$, 
$Sym^+({\mathcal G})$ plays a role of a parameter space of its left invariant metrics. The geodesic flow of any such metric on $G$ is a second order quadratic ODE system
on ${\mathcal G}$ \cite{Abr}. It is intersting to study the dependence on parameter of the  qualitative 
properties of these geodesic flows.

\subsubsection{(Locally) Homogeneous, but ``non-simply homogeneous'' spaces.} Instead of left 
invariant metrics on Lie groups, one can consider general  homogeneous spaces, say, those endowed with an isometric transitive, but not necessarily free  action of a given group $G$.  More important is the case of locally homogeneous spaces, i.e. when 
the metric varies in the space of all those locally modeled  on a fixed space $X$ endowed 
with   a (non-fixed)  $G$-invariant metric (but $G$ is fixed). Here, $G$ does not act (it acts only locally, as a pseudogroup).  As an example, we have the Robertson-Walker-Friedman-Lemaitre  spacetimes \cite{Haw, One, Wal}, which are warped products $\bar{M} = T\times_w N$, 
$\bar{g} = -dt^2 + w(t)g$, where $N$ has a constant curvature. 


 \subsubsection{Dimension $ 2+1$} So far, only the  Gauss gauge has been considered. Maybe, this is because of its ``deterministic character'', i.e. it gives rise to autonomous 
 differential equations, instead of non-autonomous ones, as in the generic case. There are however
 other situations where interesting gauges are available. As an example, in   't Hooft's  theory of systems of particles in dimension $2+1$ \cite{Hoo}, one has a flat polyhedral surface with singularities, evolving (locally) in a Minkowski space. The gauge here is fixed by the fact that time is  locally equivalent to  a ``linear time'' in the Minkowski space. In particular the time   levels
 remain flat polyhedral. 
 By consideration of  suitable spaces of such surfaces, one may be convinced there is a configuration space approach similar to  our   situation here.



\subsubsection{Non-empty spaces}

Recall that $\bar{M} = I \times M$ is a 
perfect fluid, if 
$\bar{Ric} = ({\goth p}+ \rho) dt^2 + {\goth p}\bar{g}  $, where  
${\goth p}$ is the pressure and $\rho$ is the density \cite{Haw, One, Wal}.  A Bianchi-Einstein flow can be 
defined in this case, when ${\goth p}$ and $\rho$ are functions on $TSym_n^+$, or 
more reasonably,  when they are basic functions, i.e. they  depend on the coordinate 
$q \in Sym_n^+$ alone.  
Robertson-Walker spacetimes
Êare examples of perfect fluids (strictly speaking, they would  be covered by our approach, once we consider general locally homogeneous spaces, as discussed above).

\subsubsection{Quantization of the Bianchi-Einstein flow}  We strongly believe this is
a natural case that  can be treated by a  quantum gravity theory (see for instance
\cite{Car}), that is,  a reasonable
quantization of the  Bianchi-Einstein flow  should be possible...

\subsubsection{A modified Einstein equation}  World would be perhaps  simpler if 
the Einstein equation on $TSym_n^+$ were given by the mechanical system 
determined by the Riemannian metric $\langle, \rangle$ on $Sym_n^+$  as a kinetic energy, and the 
Hilbert action ${\mathcal H} $   as a potential energy.  Recall \cite{Abr}  that solutions of such a mechanical system are curves $q(t) \in Sym_n^+$, satisfying 
$$\nabla_{q^\prime(t)} q^\prime(t) = - \nabla {\mathcal H}(q(t))$$ ($\nabla $ is the Riemannian-connection
and $\nabla {\mathcal H}$ is the Riemannian gradient of ${\mathcal H}$). Other more 
``realistic'' modified equations are  obtained by replacing the Riemannian metric 
by a pseudo-Riemannian or a Finsler one of the form $L_{\alpha, \beta}$ or 
$F_{\alpha, \beta}$ (\S \ref{geodesic.flow}).

\end{document}